\documentclass[leqno,12pt]{article} %leqno is the option to put formula numbers on the left side
\usepackage{amsmath, amssymb}
\usepackage{amsthm} %theorem environment option
\usepackage{epsfig}

\theoremstyle{plain} %text of this environment is typesetted in italics
\newtheorem{theorem}{\indent\sc Theorem}[section]
\newtheorem{lemma}[theorem]{\indent\sc Lemma}
\newtheorem{corollary}[theorem]{\indent\sc Corollary}
\newtheorem{proposition}[theorem]{\indent\sc Proposition}

\theoremstyle{definition} %text of this environment is typesetted in roman letters

\newtheorem{remark}[theorem]{\indent\sc Remark}

\newcommand\on{\operatorname}
\renewcommand\div{\on{div}}

\newcommand\Ric{\on{Ric}}

\newcommand\trace{\on{trace}}

\title{On trivial gradient hyperbolic Ricci and \\
gradient hyperbolic Yamabe solitons}

\author{Adara M. Blaga}

\date{}

\begin{document}

\maketitle

\markboth{{\small\it {\hspace{0.1cm} On trivial gradient hyperbolic Ricci and gradient hyperbolic Yamabe solitons}}}{\small\it{On trivial gradient hyperbolic Ricci and gradient hyperbolic Yamabe solitons \hspace{0.1cm}}}

%%%%%%%%%%%%%%% footnote %%%%%%%%%%%%%%%%

\begin{abstract}
We provide conditions for a compact gradient hyperbolic Ricci and a compact gradient hyperbolic Yamabe soliton to be trivial, hence, the manifold to be an Einstein manifold in the first case, and a manifold of constant scalar curvature, in the second case. In particular, we prove that for a compact gradient hyperbolic Yamabe soliton of dimension $>2$, if the second Lie derivative of the metric in the direction of the potential vector field is trace-free and divergence-free, then the above conclusion is reached.
\end{abstract}

\section{Preliminaries}

Stationary solutions to different geometric flows \cite{hip2, li, hip1}, the geometric solitons have been lately intensively studied from various points of view. Recently considered, the hyperbolic Ricci and the hyperbolic Yamabe solitons have not been extensively treated yet.

We recall that a \textit{hyperbolic Ricci soliton} is a self-similar solution of the \textit{hyperbolic Ricci flow} \cite{hi}
 $$\frac{\partial^2 g}{\partial t^2}(t)=-\Ric(t)g(t)$$
for a time-depending (semi-)Riemannian metric $g$ on a smooth manifold $M$, where $\Ric$ denotes the Ricci curvature of $(M,g)$. Therefore, a \textit{hyperbolic Ricci soliton} satisfies the equation
\begin{align}\label{e11}
\pounds_{\xi}\pounds_{\xi}g+\lambda \pounds_{\xi}g+\Ric&=\mu g
\end{align}
for a smooth vector field $\xi$ and two real scalars $\lambda$ and $\mu$, where $\pounds_{\xi}g$ is the Lie derivative of the metric $g$ into the direction of $\xi$ and $\pounds_{\xi}\pounds_{\xi}g:=\pounds_{\xi}(\pounds_{\xi}g)$. If $\xi$ is a $2$-Killing vector field, i.e., $\pounds_{\xi}\pounds_{\xi}g=0$, then a hyperbolic Ricci soliton is just a Ricci soliton.

In \cite{Adara}, we introduced the notion of \textit{hyperbolic Yamabe flow} as being an evolution equation
$$\frac{\partial^2 g}{\partial t^2}(t)=-r(t)g(t)$$
for a time-depending (semi-)Riemannian metric $g$ on a smooth manifold $M$, where $r$ denotes the scalar curvature of $(M,g)$; and its self-similar solutions, namely, the \textit{hyperbolic Yamabe solitons}, which satisfy the equation
\begin{align}\label{e1}
\pounds_{\xi}\pounds_{\xi}g+\lambda \pounds_{\xi}g&=(\mu-r)g
\end{align}
for a smooth vector field $\xi$ and two real scalars $\lambda$ and $\mu$. If $\xi$ is a $2$-Killing vector field, then a hyperbolic Yamabe soliton is just a Yamabe soliton.

In both of these cases, we shall call a soliton \textit{trivial} if its potential vector field $\xi$ is a Killing vector field, i.e., if $\pounds_{\xi}g=0$.
We notice that a trivial hyperbolic Ricci soliton is an Einstein manifold, and a trivial hyperbolic Yamabe soliton is a manifold of constant scalar curvature.

The aim of the present paper is to provide conditions for a compact gradient hyperbolic Ricci and a compact gradient hyperbolic Yamabe soliton to be a trivial soliton.

In the rest of the paper we shall consider $g$ a Riemannian metric.

\section{Triviality conditions}

Let us first remark
that a hyperbolic Ricci and a hyperbolic Yamabe soliton with potential vector field $\xi$ such that $\div(\xi)=$ constant and $\trace(\pounds_{\xi}\pounds_{\xi}g)=$ constant is a manifold of constant scalar curvature. We shall further provide sufficient conditions for a compact gradient hyperbolic Ricci and a compact gradient hyperbolic Yamabe soliton $(M,g,\xi, \lambda,\mu)$ with $\lambda\neq 0$ and $\trace(\pounds_{\xi}\pounds_{\xi}g)=0$ to be trivial. In this case, the manifold will be of constant scalar curvature.

We immediately have
\begin{proposition}
A vector field $\xi$ on a compact Riemannian manifold with $\pounds_{\xi}\pounds_{\xi}g$ trace-free and satisfying $\int_M {\Ric}(\xi,\xi)\leq 0$ is a parallel vector field.
\end{proposition}
\begin{proof}
We know \cite{AC0} that
$$\trace(\pounds_{\xi}\pounds_{\xi}g)=2\Big(\Vert \nabla {\xi}\Vert^2+\div(\nabla_{\xi}\xi)-\Ric(\xi,\xi)\Big),$$
therefore, in this case, we have
$$\Ric(\xi,\xi)=\Vert \nabla {\xi}\Vert^2+\div(\nabla_{\xi}\xi),$$
which, by integration, gives the conclusion.
\end{proof}

And we deduce
\begin{theorem}\label{c}
A compact hyperbolic Ricci and a compact hyperbolic Yamabe soliton $(M, g, \xi, \lambda,\mu)$ with $\lambda\neq 0$ such that $\pounds_{\xi}\pounds_{\xi}g$ is trace-free and
$$\int_M {\Ric}(\xi,\xi)\leq 0$$ is a trivial soliton.
\end{theorem}

We shall further provide conditions for a hyperbolic Ricci and a hyperbolic Yamabe soliton with potential vector field of gradient type to be trivial.

\bigskip

Let us firstly prove the following lemma.

\begin{lemma}\label{lemma02}
(i) If $(M^n, g, \nabla f, \lambda,\mu)$ is a gradient hyperbolic Yamabe soliton with $\lambda\neq 0$ such that $\pounds_{\nabla f}\pounds_{\nabla f}g$ is trace-free, then
\begin{equation}\label{e2}
\frac{1}{2}\Delta(\Vert \nabla f\Vert^2)=\Vert\nabla \nabla f\Vert^2+{\Ric}(\nabla f,\nabla f)-\frac{n}{2\lambda}g(\nabla f,\nabla r).
\end{equation}
Even more, we have
\begin{align*}
\frac{1}{2}\Delta(\Vert\nabla f\Vert^2)&=2\Vert\nabla \nabla f\Vert^2+\div(\nabla_{\nabla f}\nabla f)-\frac{n}{2\lambda}g(\nabla f,\nabla r)\\
&=2{\Ric}(\nabla f,\nabla f)-\div(\nabla_{\nabla f}\nabla f)-\frac{n}{2\lambda}g(\nabla f,\nabla r).
\end{align*}

(ii) If $(M^n, g, \nabla f, \lambda,\mu)$ is a gradient hyperbolic Ricci soliton with $\lambda\neq 0$ such that $\pounds_{\nabla f}\pounds_{\nabla f}g$ is trace-free, then
\begin{equation}\label{e22}
\frac{1}{2}\Delta(\Vert \nabla f\Vert^2)=\Vert\nabla \nabla f\Vert^2+{\Ric}(\nabla f,\nabla f)-\frac{1}{2\lambda}g(\nabla f,\nabla r).
\end{equation}
Even more, we have
\begin{align*}
\frac{1}{2}\Delta(\Vert \nabla f\Vert^2)&=2\Vert\nabla \nabla f\Vert^2+\div(\nabla_{\nabla f}\nabla f)-\frac{1}{2\lambda}g(\nabla f,\nabla r)\\
&=2{\Ric}(\nabla f,\nabla f)-\div(\nabla_{\nabla f}\nabla f)-\frac{1}{2\lambda}g(\nabla f,\nabla r).
\end{align*}
\end{lemma}
\begin{proof}
(i) By taking the trace into the soliton equation \eqref{e1}, we get
$$2\lambda \Delta(f)=n(\mu-r),$$
hence
$$2\lambda g(\nabla (\Delta(f)),\nabla f)=-ng(\nabla f,\nabla r),$$
which, replaced into the Bochner's formula \cite{Yano Boch}
$$\frac{1}{2}\Delta(\Vert\nabla f\Vert^2)=\Vert\nabla\nabla f\Vert^2+{\Ric}(\nabla f,\nabla f)+g(\nabla (\Delta(f)),\nabla f)$$
gives \eqref{e2}.

In a similar way, we obtain \eqref{e22}, since in the hyperbolic Ricci soliton case, we have $2\lambda \Delta(f)=n\mu-r$.
\end{proof}

Using this lemma, we prove

\begin{theorem}\label{t1}
Let $(M^n, g, \nabla f, \lambda,\mu)$ be a compact gradient hyperbolic Yamabe soliton with $\lambda\neq 0$ such that $\pounds_{\nabla f}\pounds_{\nabla f}g$ is trace-free. If 
$$\int_M {\Ric}(\nabla f,\nabla f)\geq \frac{\displaystyle n}{\displaystyle 2\lambda}\int_M g(\nabla f,\nabla r),$$ then the soliton is trivial.
\end{theorem}
\begin{proof}
By integrating the relation \eqref{e2}, we get
$$\int_M \Vert\nabla \nabla f\Vert^2=\int_M \Big(\frac{n}{2\lambda}g(\nabla f,\nabla r)-{\Ric}(\nabla f,\nabla f)\Big),$$
hence $\nabla \nabla f=0$. In this case, $\Delta(f)=0$; therefore, $f$ is a constant since the manifold is compact, and the scalar curvature $r$ is constant equal to $\mu$.
\end{proof}

Similarly, we get
\begin{theorem}\label{t2}
Let $(M, g, \nabla f, \lambda,\mu)$ be a compact gradient hyperbolic Ricci soliton with $\lambda\neq 0$ such that $\pounds_{\nabla f}\pounds_{\nabla f}g$ is trace-free. If
$$\int_M {\Ric}(\nabla f,\nabla f)\geq \frac{\displaystyle 1}{\displaystyle 2\lambda}\int_M g(\nabla f,\nabla r),$$ then the soliton is trivial.
\end{theorem}

By means of Theorems \ref{c}, \ref{t1} and \ref{t2}, we can state
\begin{corollary}
Let $(M, g, \nabla f, \lambda,\mu)$ be a compact gradient hyperbolic Ricci or a compact gradient hyperbolic Yamabe soliton with $\lambda\neq 0$ such that $\pounds_{\nabla f}\pounds_{\nabla f}g$ is trace-free. If 
$$\lambda \int_M g(\nabla f,\nabla r)\leq 0,$$ then the soliton is trivial.
\end{corollary}

Another condition for the soliton to be trivial is given in the following

\begin{theorem}
(i) Let $(M^n, g, \nabla f, \lambda,\mu)$ be a compact gradient hyperbolic Yamabe soliton with $\lambda\neq 0$ such that $\pounds_{\nabla f}\pounds_{\nabla f}g$ is trace-free.
If $$\int_M \Ric(\nabla f,\nabla f) \geq \frac{\displaystyle n^2}{\displaystyle 4\lambda^2}\int_M (\mu-r)^2,$$ then the soliton is trivial.

(ii) Let $(M^n, g, \nabla f, \lambda,\mu)$ be a compact gradient hyperbolic Ricci soliton with $\lambda\neq 0$ such that $\pounds_{\nabla f}\pounds_{\nabla f}g$ is trace-free.
If $$\int_M \Ric(\nabla f,\nabla f) \geq \frac{\displaystyle 1}{\displaystyle 4\lambda^2}\int_M (n\mu-r)^2,$$ then the soliton is trivial.
\end{theorem}
\begin{proof}
(i) Using the Bochner's formula \cite{Yano Boch}, we have
\begin{align*}
0&=\int_M\Big(\frac{1}{2}\Vert \pounds_{\nabla f}g\Vert^2-\Vert \nabla\nabla f\Vert^2-(\div(\nabla f))^2+\Ric(\nabla f,\nabla f)\Big)\\
&= \int_M \Big(\Vert \nabla \nabla f\Vert^2-(\Delta (f))^2+\Ric(\nabla f,\nabla f)\Big),
\end{align*}
which implies
\begin{align*}
\int_M \Vert \nabla \nabla f\Vert^2&=\int_M\Big((\Delta (f))^2-\Ric(\nabla f,\nabla f)\Big)\\
&=\int_M \Big(\frac{n^2(\mu-r)^2}{4\lambda^2}-\Ric(\nabla f,\nabla f)\Big)\leq 0;
\end{align*}
therefore, $\nabla \nabla f=0$.

In a similar way, we get the conclusion for the hyperbolic Ricci soliton case, since we have
$\Delta(f)=\frac{\displaystyle n\mu-r}{\displaystyle 2\lambda}$.
\end{proof}

\bigskip

If $\pounds_{\nabla f}\pounds_{\nabla f}g$ has constant trace and it's divergence-free, we obtain

\begin{lemma}\label{p1}
(i) Let $(M^n, g, \nabla f, \lambda,\mu)$ be a gradient hyperbolic Yamabe soliton with $\lambda\neq 0$ such that $\pounds_{\nabla f}\pounds_{\nabla f}g$ has constant trace and it's divergence-free. Then, for any vector field $X$ on $M$, we have
$${\Ric}(X,\nabla f)=\frac{n-1}{2\lambda}g(X,\nabla r).$$

(ii) Let $(M^n, g, \nabla f, \lambda,\mu)$ be a gradient hyperbolic Ricci soliton with $\lambda\neq 0$ such that $\pounds_{\nabla f}\pounds_{\nabla f}g$ has constant trace and it's divergence-free. Then, for any vector field $X$ on $M$, we have
$${\Ric}(X,\nabla f)=\frac{\displaystyle 1}{\displaystyle 4\lambda}g(X,\nabla r).$$
\end{lemma}
\begin{proof}
(i) We know \cite{a} that, for any vector field $X$ on $M$, we have
$$\div(\pounds_{\nabla f}g)(X)=2X(\Delta(f))+2{\Ric}(X,\nabla f).$$
By taking the divergence into the soliton equation \eqref{e1}, we get
$$\div(\pounds_{\nabla f}g)(X)=-\frac{g(X,\nabla r)}{\lambda},$$
therefore,
$$g(X,\nabla (\Delta(f)))+{\Ric}(X,\nabla f)=-\frac{1}{2\lambda}g(X,\nabla r).$$
But $$\Delta(f)=\frac{n(\mu-r)}{2\lambda}-\frac{1}{2\lambda}\trace(\pounds_{\nabla f}\pounds_{\nabla f}g),$$
hence
$$g(X,\nabla (\Delta(f)))=-\frac{n}{2\lambda}g(X,\nabla r),$$
and we get
$${\Ric}(X,\nabla f)=\frac{n-1}{2\lambda}g(X,\nabla r),$$
hence the conclusion.

In a similar way, we get the conclusion for the hyperbolic Ricci soliton case, by using Schur's Lemma, $\div(\Ric)=\frac{\displaystyle dr}{\displaystyle 2}$.
\end{proof}

\begin{remark}
The result form Lemma \ref{p1} holds true if we add the connectness condition on the manifold, keep $\lambda\neq 0$, and ask for $\pounds_{\nabla f}\pounds_{\nabla f}g$ to be only divergence-free.
\end{remark}

The above conditions lead also to the following result
\begin{proposition}
Let $(M^n, g, \nabla f, \lambda,\mu)$ be a compact and connected gradient hyperbolic Ricci or a compact and connected gradient hyperbolic Yamabe soliton with $\lambda\neq 0$ such that $\pounds_{\nabla f}\pounds_{\nabla f}g$ is divergence-free. If
$$\lambda\int_M{\Ric}(\nabla f, \nabla r)\leq 0,$$
then the manifold is of constant scalar curvature.
\end{proposition}
\begin{proof}
It follows from Lemma \ref{p1}, taking into account that
$$\lambda{\Ric}(\nabla f, \nabla r)=\frac{n-1}{2}\Vert \nabla r\Vert^2$$
in the hyperbolic Yamabe case, and
$$\lambda{\Ric}(\nabla f, \nabla r)=\frac{1}{4}\Vert \nabla r\Vert^2$$
in the hyperbolic Ricci case.
\end{proof}

From Lemmas \ref{lemma02} and \ref{p1}, by a direct computation, we get

\begin{proposition}\label{p2}
(i) If $(M^n, g, \nabla f, \lambda,\mu)$ is a gradient hyperbolic Yamabe soliton with $\lambda\neq 0$ such that $\pounds_{\nabla f}\pounds_{\nabla f}g$ is trace-free and divergence-free, then
\begin{align*}
\frac{1}{2}\Delta(\Vert \nabla f\Vert^2)&=\frac{n-2}{n-1}\Vert \nabla \nabla f\Vert^2-\frac{1}{n-1}\div(\nabla_{\nabla f}\nabla f).
\end{align*}

(ii) If $(M^n, g, \nabla f, \lambda,\mu)$ is a gradient hyperbolic Ricci soliton with $\lambda\neq 0$ such that $\pounds_{\nabla f}\pounds_{\nabla f}g$ is trace-free and divergence-free, then
\begin{align*}
\frac{1}{2}\Delta(\Vert \nabla f\Vert^2)&=-\div(\nabla_{\nabla f}\nabla f).
\end{align*}
\end{proposition}

Now we can prove
\begin{theorem}
Let $(M^n, g, \nabla f, \lambda,\mu)$ be a compact gradient hyperbolic Yamabe soliton with $\lambda\neq 0$ such that $\pounds_{\nabla f}\pounds_{\nabla f}g$ is trace-free and divergence-free.
If $n>2$, then the soliton is trivial.
\end{theorem}
\begin{proof}
By integrating the relation from Proposition \ref{p2} (i), we get
$$\int_M \Vert \nabla \nabla f\Vert^2=0,$$
hence the conclusion.
\end{proof}

\vspace{2mm} \noindent \footnotesize
\begin{minipage}[b]{10cm}
Adara M. Blaga \\
Department of Mathematics \\
Faculty of Mathematics and Computer Science \\
West University of Timi\c{s}oara, Romania \\
Email: adarablaga@yahoo.com
\end{minipage}

\end{document}